\documentclass[11pt,reqno]{amsart}

\usepackage{amscd,amssymb,amsmath,amsthm}
\usepackage[arrow,matrix]{xy}
\usepackage{graphicx}
\usepackage{color}
\usepackage{cite}
\newtheorem{thm}{Theorem}
\newtheorem{defn}{Definition}
\newtheorem{lemma}{Lemma}
\newtheorem{pro}{Proposition}
\newtheorem{rk}{Remark}
\newtheorem{cor}{Corollary}

\numberwithin{equation}{section} 

\newcommand{\M}{{\mathcal M}}
\newcommand{\A}{{\mathcal A}}

\newcommand{\bea}{\begin{eqnarray}}
\newcommand{\eea}{\end{eqnarray}}

\newcommand{\R}{\mathbb{R}}

\def\M{\mathcal M}

\def\R{\mathbb{R}}

\begin{document}
\title [On real chains of evolution algebras]
{On real chains of evolution algebras}

\author { B.A. Omirov, U.A. Rozikov, K. M. Tulenbayev}

 \address{B.\ A.\ Omirov and U.\ A.\ Rozikov\\ Institute of Mathematics, National University of Uzbekistan, Tashkent, Uzbekistan.}
 \email {omirovb@mail.ru \ \ rozikovu@yandex.ru}
\address{K.\ M.\ Tulenbayev\\ Suleyman Demirel University, Kazakhstan.}
 \email {tulen75@hotmail.com}

\begin{abstract} In this paper we define a chain of $n$-dimensional evolution algebras
corresponding to a permutation of $n$ numbers. We show that a chain of evolution algebras (CEA) corresponding to a permutation is trivial (consisting only algebras with zero-multiplication) iff the permutation has not a fixed point. We show that a CEA is a chain of nilpotent algebras (independently on time) iff it is trivial. We construct a wide class of chains of 3-dimensional EAs and a class of symmetric $n$-dimensional CEAs.
A construction of arbitrary dimensional CEAs is given. Moreover, for a chain of 3-dimensional EAs we study the behavior of the baric
property, the behavior of the set of absolute nilpotent elements and dynamics of the set of idempotent elements depending on the time.

{\it AMS classifications (2010):} 17D92; 17D99; 60J27\\[2mm]

{\it Keywords:} Evolution algebra; time; Chapman-Kolmogorov
equation; nilpotentcy.
\end{abstract}
\maketitle
\section{Introduction} \label{sec:intro}

Evolutionary theory is important to a proper understanding of living populations
at all levels. The paper \cite{K} gives a comprehensive survey of problems in evolutionary theory studied by mathematical methods, both with regard to past developments, the present state of the subject and probable future trends. Many of the models considered are actually within the scope of mathematical population genetics, although the general ideas discussed have implications for a much wider area.

The formal language of abstract algebra to study of genetics was introduced in papers \cite{e1,e2,e3}. In recent years many
authors have tried to investigate the difficult problem of
classification of these algebras. The most comprehensive references
for the mathematical research done in this area are \cite{clor,ly,m,t,w}.

 In \cite{ly} an evolution algebra $\mathcal A$ associated with
free population is introduced and using this non-associative algebra
many results are obtained in explicit form, e.g. the explicit
description of stationary quadratic operators, and the explicit
solutions of a nonlinear evolutionary equation in the absence of
selection, as well as general theorems on convergence to equilibrium
in the presence of selection.

In \cite{t} a new type of evolution algebra is introduced. This
evolution algebra is defined as follows.
 Let $(E,\cdot)$ be an algebra over a field $K$. If it admits a
basis $\{e_1,e_2,\dots\}$ such that $e_i\cdot e_j=0$, if $i\ne j$ and $
e_i\cdot e_i=\sum_{k}a_{ik}e_k$ for any $i$, then this algebra is
called an {\it evolution algebra}.

In this paper by the term {\it evolution algebra} we will understand
a finite dimensional evolution algebra $E$ (as mentioned above) over the
field $\R$.

In the book \cite{t}, the foundation of evolution algebra theory and
applications in non-Mendelian genetics and Markov chains are
developed, with pointers to some further research topics.

In \cite{rt} the algebraic structures of function spaces defined by
graphs and state spaces equipped with Gibbs measures by associating
evolution algebras are studied. Results of \cite{rt} also allow  a
natural introduction of thermodynamics in studying of several
systems of biology, physics and mathematics by theory of evolution
algebras.x

Recently in \cite{CLR} a notion of a chain of evolution algebras (CEA) is introduced.
This chain is a dynamical system the state of which at each given time is an evolution algebra.
The sequence of matrices of the structural constants for this
chain of evolution algebras satisfies the Chapman-Kolmogorov equation. It is known that if a matrix satisfying
the Chapman-Kolmogorov equation is stochastic, then it generates a Markov process. In \cite{ht} the authors investigated the time evolution of stochastic non-Markov processes as they occur in the coarse-grained description of open and closed systems.  Some aspects of the theory are illustrated for the two-state process and the Gauss process. The CEA also can be motivated in a such way.

Following \cite{CLR} we recall the definitions. Consider a family $\left\{E^{[s,t]}:\ s,t \in \R,\ 0\leq s\leq t
\right\}$ of $n$-dimensional evolution algebras over the field $\R$
with basis $\{e_1,\dots,e_n\}$ and multiplication table
\begin{equation}\label{1}
 e_ie_i = \sum_{j=1}^na_{ij}^{[s,t]}e_j, \ \ i=1,\dots,n; \ \
e_ie_j =0,\ \ i\ne j,\end{equation} where parameters $s,t$ are
considered as time.

Denote by
$\M^{[s,t]}=\left(a_{ij}^{[s,t]}\right)_{i,j=1,\dots,n}$-the matrix
of structural constants of $E^{[s,t]}$.

\begin{defn}\label{d1} A family $\left\{E^{[s,t]}:\ s,t \in \R,\ 0\leq s\leq t
\right\}$ of $n$-dimensional evolution algebras over the field $\R$
is called a chain of evolution algebras if the matrix
$\M^{[s,t]}$ of its structural constants satisfies the
Chapman-Kolmogorov equation:
\begin{equation}\label{k2}
\M^{[s,t]}=\M^{[s,\tau]}\M^{[\tau,t]}, \ \ \mbox{for any} \ \
s<\tau<t.
\end{equation}
\end{defn}

\begin{defn}\label{d2} A CEA is called a time-homogenous CEA if
the matrix $\M^{[s,t]}$ depends only on $t-s$. In this case we write
$\M^{[t-s]}$.
\end{defn}

In \cite{RM} 25 distinct examples of chains of two-dimensional evolution algebras are constructed.
 For all of chains constructed in \cite{CLR} and \cite{RM} the behavior of the baric
property, the behavior of the set of absolute nilpotent elements and dynamics of the set of idempotent elements depending on the time are studied.

The paper is organized as follows. In Section 2 we define a chain of $n$-dimensional evolution algebras
corresponding to a permutation of $n$ numbers. We show that a CEA corresponding to a permutation is trivial (consisting only algebras with zero-products) iff the permutation has not a fixed point. In Section 3 we show that a CEA is a chain of nilpotent algebras (independently on time) iff it is trivial. Section 4 is devoted to a construction of chains of 3-dimensional EAs. In Section 5 for a chain of 3-dimensional EAs we study the behavior of the baric
property, the behavior of the set of absolute nilpotent elements and dynamics of the set of idempotent elements depending on the time. Section 6 is devoted to construction of an arbitrary dimensional symmetric CEAs. In the last section we give an argument to construct a wide class of chains of arbitrary finite dimensional EAs.

 \section{The CEA corresponding to a permutation}

Let $S_n$ be the group of permutations of the set  $\{1,2,\dots,n\}$.
For a quadratic matrix $A=\{a_{ij}\}_{i,j=1}^n$ we say that it corresponds
to a permutation $\pi$ if
\begin{equation}\label{pi}
a_{ij}=\left\{\begin{array}{ll}
0, \ \ \mbox{if} \ \ j\ne \pi(i)\\[2mm]
\ne 0,  \ \ \mbox{if} \ \ j=\pi(i).
\end{array}\right.
\end{equation}

Let $A_\pi$, $B_\pi$ be two matrices corresponding to a permutation $\pi$. It is easy to see that
\begin{equation}\label{p2}
A_\pi B_\pi=C_{\pi^2}.
\end{equation}

For a fixed $\pi\in S_n$ we shall construct CEAs corresponding
to $\M^{[s,t]}_\pi$.

\begin{thm}\label{t1} Let $\pi\in S_n$ and $\left\{\M^{[s,t]}_\pi, \ \ 0\leq s\leq t\right\}$ be a family of
matrices which satisfies the equation (\ref{k2}). Then
\begin{itemize}

\item[(i)] $\M^{[s,t]}_\pi\equiv 0$
if $\{i: \pi(i)=i\}= \emptyset$.

\item[(ii)] $\M^{[s,t]}_\pi\ne 0$ if $\{i: \pi(i)=i\}\ne \emptyset$.
\end{itemize}
\end{thm}
\begin{proof} (i) Let $\M^{[s,\tau]}_\pi\M^{[\tau,t]}_\pi=(c_{ij})_{i,j=1}^n$. We have by (\ref{p2}) that
\begin{equation}\label{pii}
c_{ij}=\left\{\begin{array}{ll}
0, \ \ \mbox{if} \ \ j\ne \pi^2(i)\\[2mm]
\ne 0,  \ \ \mbox{if} \ \ j=\pi^2(i).
\end{array}\right.
\end{equation}
From (\ref{k2}) we get
\begin{equation}\label{cp}\left\{\begin{array}{ll}
0\ne c_{i\pi^2(i)}= a_{i\pi(i)}^{[s,\tau]}a_{\pi(i)\pi(\pi(i))}^{[\tau,t]}=a_{i\pi(i)}^{[s,t]},  \\[2mm]
a_{ij}^{[s,t]}=0,  \ \ \mbox{if} \ \ j\ne \pi(i),  \ \ \forall s,t.
\end{array}\right.
\end{equation}

Hence $a_{i\pi(i)}^{[s,t]}=0$ iff $\pi(i)\ne \pi^2(i)$, i.e., $i\ne \pi(i)$. \\

(ii) Now assume $\{i: \pi(i)=i\}\ne \emptyset$. Without loss of generality we can take $\{i: \pi(i)=i\}=\{1,2,\dots, k\}$.
Then from (\ref{cp}) we obtain
$$a_{jj}^{[s,\tau]}a_{jj}^{[\tau,t]}=a_{jj}^{[s,t]}, \ \ j=1,\dots, k.$$
This equation is known as Cantor's second equation which has
very rich family of solutions:

a) $a_{jj}^{[s,t]}={\Phi_j(t)\over \Phi_j(s)}$, where $\Phi_j$ is an arbitrary
function with $\Phi_j(s)\ne 0$;

b) $$a_{jj}^{[s,t]}=\left\{\begin{array}{ll}
1, \ \ \mbox{if}\ \ s\leq t<a_j,\\[2mm]
0, \ \ \mbox{if} \ \ t\geq a_j,\\
\end{array}\right. \ \ \mbox{where} \ \ a_j>0.$$
The CEAs $E_{\pi}^{[s,t]}$ corresponding to solutions a)-b) have matrix of structural constants in the following form:
$$\M_\pi^{[s,t]}=\left(\begin{array}{cccccc}
a_{11}^{[s,t]} & 0 &0& \dots & \dots &0\\[2mm]
0 & a_{22}^{[s,t]} &0& \dots & \dots &0\\[3mm]
\vdots&\dots& \dots& \dots& &\vdots\\[3mm]
0&0&\dots&a_{kk}^{[s,t]}&\dots&0\\[2mm]
0&0&\dots&0&\dots&0\\[2mm]
\vdots&\dots& \dots& \dots& &\vdots\\[2mm]
0&0&\dots&0&\dots&0
\end{array}\right).$$
This completes the proof.
\end{proof}

\section{Nilpotent CEAs}

\begin{defn} An element $x$ of an algebra $\mathcal A$  is called nil if
there exists $n(a) \in \mathbb{N}$ such that $(\cdots
\underbrace{((x\cdot x)\cdot x)\cdots x}_{n(a)})=0$.
Algebra $\mathcal A$ is called nil if every  element of the algebra
is nil.
\end{defn}

For $k\geq 1$ we introduce the following sequences:
$$\mathcal A^{(1)} =\mathcal A^, \ \ \mathcal A^{(k+1)} = \mathcal A^{(k)}\mathcal A^{(k)}.$$
$$\mathcal A^{<1>} = \mathcal A, \ \ \mathcal A^{<k+1>} = \mathcal A^{<k>}\mathcal A.$$
$$\mathcal A^1 = \mathcal A, \ \ \mathcal A^k =\sum_{i=1}^{k-1}\mathcal A^i\mathcal A^{k-i}.$$

\begin{defn} An algebra $\mathcal A$ is called
\begin{itemize}
\item[(i)] solvable if there exists $n\in \mathbb{N}$ such that $\mathcal A^{(n)} = 0$ and the minimal such number is
called index of solvability;
\item[(ii)] right nilpotent if there exists $n\in \mathbb{N}$ such that $\mathcal A^{<n>} = 0$ and the minimal such
number is called index of right nilpotency;
\item[(iii)] nilpotent if there exists $n\in \mathbb{N}$ such that $\mathcal A^n = 0$ and the minimal such number
is called index of nilpotency.
\end{itemize}
\end{defn}

The following theorem is known
\begin{thm} \label{thm2}\cite{clor}, \cite{Com} The following statements are equivalent for
an $n$-dimensional evolution algebra $E$:

a) The matrix corresponding to $E$ can be written as
\begin{equation}\label{lor5}
\widehat{A}=\left(
  \begin{array}{ccccc}
 0 & a_{12} & a_{13} &\dots &a_{1n} \\[1.5mm]
 0 & 0 & a_{23} &\dots &a_{2n} \\[1.5mm]
0 & 0 & 0 &\dots &a_{3n} \\[1.5mm]
\vdots & \vdots & \vdots &\cdots & \vdots \\[1.5mm]
0 & 0 & 0 &\cdots &0 \\
\end{array}
\right);
\end{equation}

b) $E$ is a right nilpotent algebra;

c) $E$ is a nil algebra.

d) $E$ is a nilpotent algebra.
\end{thm}

This section is devoted to the following
\begin{thm} A chain of evolution algebras $E^{[s,t]}$ is nilpotent for any
$(s,t)$ if and only if $\M^{[s,t]}=0$ for any $(s,t)$.
\end{thm}
\begin{proof} {\sl Necessity.} Let
\begin{equation}\label{cl}
\M^{[s,t]}=\left(
  \begin{array}{ccccc}
 0 & a^{[s,t]}_{12} & a^{[s,t]}_{13} &\dots &a^{[s,t]}_{1n} \\[1.5mm]
 0 & 0 & a^{[s,t]}_{23} &\dots &a^{[s,t]}_{2n} \\[1.5mm]
0 & 0 & 0 &\dots &a^{[s,t]}_{3n} \\[1.5mm]
\vdots & \vdots & \vdots &\cdots & \vdots \\[1.5mm]
0 & 0 & 0 &\cdots &0 \\
\end{array}
\right).
\end{equation}
From equation (\ref{k2}) we get
\begin{equation}\label{cl1}
\M^{[s,t]}=\left(
  \begin{array}{cccccc}
 0 & 0 & a^{[s,t]}_{13} &a^{[s,t]}_{14} &\dots &a^{[s,t]}_{1n} \\[1.5mm]
 0 & 0 & 0 &a^{[s,t]}_{24} &\dots &a^{[s,t]}_{2n} \\[1.5mm]
0 & 0 & 0 &0&\dots &a^{[s,t]}_{3n} \\[1.5mm]
\vdots & \vdots & \vdots &\cdots & \vdots \\[1.5mm]
0 & 0 & 0 &0&\cdots &0 \\
\end{array}
\right).
\end{equation}
Iterating we get $\M^{[s,t]}=0$.

{\sl Sufficiency.} Straightforward.

\end{proof}

\section{A construction of chains of three-dimensional EAs}

In this section we shall construct several new chains of three-dimensional evolution algebras.
Consider a matrix of structural constants having the form
\begin{equation}\label{M}
\M^{[s,t]}=\left(\begin{array}{ccc}
a_{1}^{[s,t]} &a_{2}^{[s,t]} &a_{3}^{[s,t]}\\[2mm]
0&a_{4}^{[s,t]} &a_{5}^{[s,t]}\\[2mm]
0 & 0 & a_{6}^{[s,t]}
\end{array}\right).
\end{equation}
Then the equation (\ref{k2}) becomes:
\begin{equation}\label{3a}\begin{array}{llllll}
a_{1}^{[s,\tau]}a_{1}^{[\tau,t]}=a_{1}^{[s,t]}\\[2mm]
a_{1}^{[s,\tau]}a_{2}^{[\tau,t]}+a_{2}^{[s,\tau]}a_{4}^{[\tau,t]}=a_{2}^{[s,t]}\\[2mm]
a_{1}^{[s,\tau]}a_{3}^{[\tau,t]}+a_{2}^{[s,\tau]}a_{5}^{[\tau,t]}+a_{3}^{[s,\tau]}a_{6}^{[\tau,t]}=a_{3}^{[s,t]}\\[2mm]
a_{4}^{[s,\tau]}a_{4}^{[\tau,t]}=a_{4}^{[s,t]}\\[2mm]
a_{4}^{[s,\tau]}a_{5}^{[\tau,t]}+a_{5}^{[s,\tau]}a_{6}^{[\tau,t]}=a_{5}^{[s,t]}\\[2mm]
a_{6}^{[s,\tau]}a_{6}^{[\tau,t]}=a_{6}^{[s,t]}.
\end{array}
\end{equation}
We note that the first, the fourth and the last equations of the system (\ref{3a}) are Cantor's second equations and they have the following solutions  ($i=1,4,6$):

($1_i$)
\begin{equation}\label{146}
a_{i}^{[s,t]}={\Phi_i(t)\over \Phi_i(s)},
\end{equation}
 where $\Phi_i$ is an arbitrary
function with $\Phi_i(s)\ne 0$, $i=1,4,6$;

($2_i$)
\begin{equation}\label{b146}
a_{i}^{[s,t]}=\left\{\begin{array}{ll}
1, \ \ \mbox{if}\ \ s\leq t<\alpha_i,\\[2mm]
0, \ \ \mbox{if} \ \ t\geq \alpha_i,\\
\end{array}\right. \ \ \mbox{where} \ \ \alpha_i>0, \, i=1,4,6.
\end{equation}
Thus we have eight possibilities for $(a_{1}^{[s,t]}, a_{4}^{[s,t]}, a_{6}^{[s,t]})$.
Consider these possibilities:

{\bf Case} $(1_1), (1_4), (1_6)$: The second equation of (\ref{3a}) becomes
$${\Phi_1(\tau)\over \Phi_1(s)}a_2^{[\tau,t]}+{\Phi_4(t)\over \Phi_4(\tau)}a_2^{[s,\tau]}=a_2^{[s,t]}.$$
From this equation denoting $\gamma(s,t)= {\Phi_1(s)\over \Phi_4(t)}a_2^{[s,t]}$ we get
$$\gamma(s,t)=\gamma(s,\tau)+\gamma(\tau,t).$$
This equation is known as Cantor's first equation which has solutions
$$\gamma(s,t)=\xi(t)-\xi(s),  \ \ \mbox{where} \ \ \xi \ \ \mbox{is an arbitrary function}.$$
Hence
\begin{equation}\label{2}
a_2^{[s,t]}={\Phi_4(t)\over \Phi_1(s)}(\xi(t)-\xi(s)).
\end{equation}
Similarly, one gets that
\begin{equation}\label{5}
a_5^{[s,t]}={\Phi_6(t)\over \Phi_4(s)}(f(t)-f(s)),
\end{equation}
where $f$ is an arbitrary function.
Using these functions from the third equation of (\ref{3a}) we obtain
\begin{equation}\label{3b}
\varphi^{[s,\tau]}+\varphi^{[\tau,t]}=\varphi^{[s,t]}+(f(\tau)-f(t))(\xi(\tau)-\xi(s)),
\end{equation}
where $\varphi^{[s,t]}={\Phi_1(s)\over \Phi_6(t)}a_3^{[s,t]}$. We shall find solutions to (\ref{3b}) which have the form
$$\varphi^{[s,t]}=af(s)\xi(s)+bf(s)\xi(t)+cf(t)\xi(s)+df(t)\xi(t).$$
Substituting this function in (\ref{3b}) we obtain $b=0$, $c=-1$ and $a+d=1$. Hence the following function is a solution to (\ref{3b}):
$$\varphi_0^{[s,t]}=af(s)\xi(s)-f(t)\xi(s)+(1-a)f(t)\xi(t),$$
where $a\in \R$. Moreover, it is easy to see that $\varphi^{[s,t]}=\varphi_0^{[s,t]}+\gamma(t)-\gamma(s)$ is also solution to (\ref{3b}), where $\gamma$ is an arbitrary function.

Consequently,
\begin{equation}\label{sa3}
a_3^{[s,t]}={\Phi_6(t)\over \Phi_1(s)}\left[af(s)\xi(s)-f(t)\xi(s)+(1-a)f(t)\xi(t)+\gamma(t)-\gamma(s)\right].
\end{equation}

Thus we have proved the following

\begin{pro}\label{pr1} The matrix (\ref{M}) with elements given by (\ref{146}), (\ref{2}), (\ref{5}) and (\ref{sa3}) generates a CEA.
\end{pro}

{\bf Case} $(1_1), (1_4), (2_6)$: In this case the functions $a_i^{[s,t]}$, $i=1,2,4$
 coincide with functions given in the previous case. The function $a_6^{[s,t]}$ is given by (\ref{b146}). Here we have to find $a_3^{[s,t]}$ and $a_5^{[s,t]}$.
From (\ref{3a}) we have
 \begin{equation}\label{55}
 a_5^{[s,t]}=\left\{\begin{array}{ll}
 a_5^{[s,\tau]}+{\Phi_4(\tau)\over \Phi_4(s)}a_5^{[\tau,t]}, \ \ \mbox{if} \ \ s\leq t<\alpha_6\\[4mm]
 {\Phi_4(\tau)\over \Phi_4(s)}a_5^{[\tau,t]}, \ \ \mbox{if} \ \  t\geq \alpha_6.
 \end{array}\right.
  \end{equation}
 Denoting $\psi^{[s,t]}=\Phi_4(s)a_5^{[s,t]}$ from (\ref{55}) we get
 $$
 \psi^{[s,t]}=\left\{\begin{array}{ll}
 \psi^{[s,\tau]}+\psi^{[\tau,t]}, \ \ \mbox{if} \ \ s\leq t< \alpha_6\\[4mm]
 \psi^{[\tau,t]}, \ \ \mbox{if} \ \  t\geq \alpha_6.
 \end{array}\right.
  $$
This equation has the following solution
$$
 \psi^{[s,t]}=\left\{\begin{array}{ll}
 \beta(t)-\beta(s), \ \ \mbox{if} \ \ s\leq t< \alpha_6\\[4mm]
 \gamma(t), \ \ \mbox{if} \ \  t\geq \alpha_6,
 \end{array}\right.
 $$
  where $\beta$ and $\gamma$ are arbitrary functions. Therefore, we have
  \begin{equation}\label{sa55}
  a_5^{[s,t]}={1\over \Phi_4(s)}\left\{\begin{array}{ll}
\beta(t)-\beta(s), \ \ \mbox{if} \ \ s\leq t<\alpha_6\\[4mm]
 \gamma(t), \ \ \mbox{if} \ \  t\geq \alpha_6,
 \end{array}\right.
  \end{equation}

 Now for $a_3^{[s,t]}$ we have
 \begin{equation}\label{33}
 a_3^{[s,t]}={1\over \Phi_1(s)}\left\{\begin{array}{ll}
 \Phi_1(s)a_3^{[s,\tau]}+\Phi_1(\tau)a_3^{[\tau,t]}+(\xi(\tau)-\xi(s))(\beta(t)-\beta(\tau)), \ \ \mbox{if} \ \ s\leq t< \alpha_6\\[4mm]
 \Phi_1(\tau)a_3^{[\tau,t]}+(\xi(\tau)-\xi(s))\gamma(t), \ \ \mbox{if} \ \  t\geq \alpha_6.
 \end{array}\right.
  \end{equation}
 Denoting $\eta(s,t)=\Phi_1(s)a_3^{[s,t]}$ and using arguments of the previous case we get
 $$\eta(s,t)=\left\{\begin{array}{ll}
 a\beta(s)\xi(s)-\beta(t)\xi(s)+(1-a)\beta(t)\xi(t)+\theta(t)-\theta(s), \ \ \mbox{if} \ \ s\leq t< \alpha_6\\[4mm]
\gamma(t)(v(t)+b\xi(t)-\xi(s)), \ \ \mbox{if} \ \  t\geq \alpha_6,
\end{array}\right.
$$
where $b\in \R$ and $\theta$, $v$ are arbitrary functions.
Consequently,
\begin{equation}\label{s33}
a_3^{[s,t]}={1\over \Phi_1(s)}\left\{\begin{array}{ll}
 a\beta(s)\xi(s)-\beta(t)\xi(s)+(1-a)\beta(t)\xi(t)+\theta(t)-\theta(s), \ \ \mbox{if} \ \ s\leq t< \alpha_6\\[4mm]
\gamma(t)(v(t)+b\xi(t)-\xi(s)), \ \ \mbox{if} \ \  t\geq \alpha_6,
\end{array}\right.
\end{equation}

Thus we have proved the following

\begin{pro} The matrix (\ref{M}) with elements given by (\ref{146}) for $i=1,4$; (\ref{b146}) for $i=6$; (\ref{2}), (\ref{sa55}) and (\ref{s33}) generates a CEA.
\end{pro}

{\bf Case} $(1_1), (2_4), (2_6)$: From the second equation of (\ref{3a}) we get
 \begin{equation}\label{sa22}
  a_2^{[s,t]}={1\over \Phi_1(s)}\left\{\begin{array}{ll}
\omega(t)-\omega(s), \ \ \mbox{if} \ \ s\leq t< \alpha_4\\[4mm]
 \delta(t), \ \ \mbox{if} \ \  t\geq \alpha_4,
 \end{array}\right.
  \end{equation}
where $\omega$ and $\delta$ are arbitrary functions.

Now we shall find $a_5^{[s,t]}$. Without loss of generality we assume $\alpha_4<\alpha_6$.
Then we obtain the following equation
$$
  a_5^{[s,t]}=\left\{\begin{array}{lll}
   a_5^{[s,\tau]}+ a_5^{[\tau,t]}, \ \ \mbox{if} \ \ s\leq \tau<\min\{\alpha_4,t\}\\[3mm]
a_5^{[s,\tau]}, \ \ \mbox{if} \ \ \alpha_4\leq \tau<t< \alpha_6\\[3mm]
0, \ \ \mbox{if} \ \ \alpha_4\leq \tau, \ \ \alpha_6\leq t.
\end{array}\right.
$$
This equation has the following solutions
\begin{equation}\label{se5}
  a_5^{[s,t]}=\left\{\begin{array}{lll}
   \zeta(t)-\zeta(s), \ \ \mbox{if} \ \ s\leq \tau<\min\{\alpha_4,t\}\\[3mm]
p(s), \ \ \mbox{if} \ \ \alpha_4\leq \tau<t< \alpha_6\\[3mm]
0, \ \ \mbox{if} \ \ \alpha_4\leq \tau, \ \ \alpha_6\leq t,
\end{array}\right.
\end{equation}
where $\zeta$ and $p$ are arbitrary functions.

In this case $a_3^{[s,t]}$ satisfies the following equation
$$
  a_3^{[s,t]}=\left\{\begin{array}{lll}
a_3^{[s,\tau]}+ {\Phi_1(\tau)\over \Phi_1(s)}a_3^{[\tau,t]}+{1\over \Phi_1(s)}(\omega(\tau)-\omega(s))(\zeta(t)-\zeta(\tau)), \ \ \mbox{if} \ \ s\leq \tau<\min\{\alpha_4,t\}\\[3mm]
a_3^{[s,\tau]}+{\Phi_1(\tau)\over \Phi_1(s)}a_3^{[\tau,t]}+{1\over \Phi_1(s)}\delta(\tau)p(\tau), \ \ \mbox{if} \ \ \alpha_4\leq \tau<t< \alpha_6\\[3mm]
{\Phi_1(\tau)\over \Phi_1(s)}a_3^{[\tau,t]}, \ \ \mbox{if} \ \ \alpha_4\leq \tau, \ \ \alpha_6\leq t.
\end{array}\right.
$$
Similarly as above one can see that this equation has the following solution
\begin{equation}\label{s433}
a_3^{[s,t]}={1\over \Phi_1(s)}\left\{\begin{array}{lll}
 a\zeta(s)\omega(s)-\zeta(t)\omega(s)+(1-a)\zeta(t)\omega(t)+d(t)-d(s), \ \ \mbox{if} \ \ s\leq \tau<\min\{\alpha_4,t\}\\[4mm]
a\delta(s)p(s)-(1+a)\delta(t)p(t)+e(t)-e(s), \ \ \mbox{if} \ \  \alpha_4\leq \tau<t< \alpha_6,\\[4mm]
m(t), \ \ \mbox{if} \ \  \alpha_4\leq \tau, \ \ \alpha_6\leq t,
\end{array}\right.
\end{equation}
where $d, e$ and $m$ are arbitrary functions, $a\in \R$.

Thus we have proved the following

\begin{pro} The matrix (\ref{M}) with elements given by (\ref{146}) for $i=1$; (\ref{b146}) for $i=4,6$; (\ref{sa22}), (\ref{se5}) and (\ref{s433}) generates a CEA.
\end{pro}

{\bf Case} $(2_1), (2_4), (2_6)$: Assume $\alpha_1<\alpha_4<\alpha_6$. In this case for $a_2^{[s,t]}$ we have
$$
  a_2^{[s,t]}=\left\{\begin{array}{lll}
a_2^{[s,\tau]}+a_2^{[\tau,t]}, \ \ \mbox{if} \ \ s\leq \tau<\min\{\alpha_1,t\}\\[3mm]
a_2^{[s,\tau]}, \ \ \mbox{if} \ \ \alpha_1\leq \tau<t< \alpha_4\\[3mm]
0, \ \ \mbox{if} \ \ \alpha_1\leq \tau, \ \ \alpha_4\leq t.
\end{array}\right.
$$
This has the following solution
\begin{equation}\label{lw}
  a_2^{[s,t]}=\left\{\begin{array}{lll}
l(t)-l(s), \ \ \mbox{if} \ \ s\leq \tau<\min\{\alpha_1,t\}\\[3mm]
w(s), \ \ \mbox{if} \ \ \alpha_1\leq \tau<t<\alpha_4\\[3mm]
0, \ \ \mbox{if} \ \ \alpha_1\leq \tau, \ \ \alpha_4\leq t,
\end{array}\right.
\end{equation}
where $l$ and $w$ are arbitrary functions.

Similarly
\begin{equation}\label{5lw}
  a_5^{[s,t]}=\left\{\begin{array}{lll}
A(t)-A(s), \ \ \mbox{if} \ \ s\leq \tau<\min\{\alpha_4,t\}\\[3mm]
B(s), \ \ \mbox{if} \ \ \alpha_4\leq \tau<t<\alpha_6\\[3mm]
0, \ \ \mbox{if} \ \ \alpha_4\leq \tau, \ \ \alpha_6\leq t,
\end{array}\right.
\end{equation}
where $A$ and $B$ are arbitrary functions.
For $a_3^{[s,t]}$ we have the following equation
$$
  a_3^{[s,t]}=\left\{\begin{array}{llll}
a_3^{[s,\tau]}+a_3^{[\tau,t]}+(l(\tau)-l(s))(A(t)-A(\tau)), \ \ \mbox{if} \ \ s\leq \tau<\min\{\alpha_1,t\}\\[3mm]
a_3^{[s,\tau]}+w(s)(A(t)-A(\tau)), \ \ \mbox{if} \ \ \alpha_1\leq \tau<t<\alpha_4\\[3mm]
a_3^{[\tau,t]}, \ \ \mbox{if} \ \ \alpha_4\leq t<\alpha_6\\[3mm]
0, \ \ \mbox{if} \ \ t\geq \alpha_6.
\end{array}\right.
$$
This equation has the following solution
\begin{equation}\label{s633}
a_3^{[s,t]}=\left\{\begin{array}{lll}
 aA(s)l(s)-A(t)l(s)+(1-a)A(t)l(t)+c(t)-c(s), \ \ \mbox{if} \ \ s\leq \tau<\min\{\alpha_1,t\}\\[4mm]
(A(t)+bA(s)+g(s))w(s), \ \ \mbox{if} \ \  \alpha_1\leq \tau<t<\alpha_4,\\[4mm]
q(s), \ \ \mbox{if} \ \  \alpha_4\leq t< \alpha_6\\[4mm]
0, \ \ \mbox{if} \ \  t\geq \alpha_6,
\end{array}\right.
\end{equation}
where $c, g$ and $q$ are arbitrary functions, $a,b\in \R$.

Thus we have proved the following

\begin{pro} The matrix (\ref{M}) with elements given by (\ref{b146}) for $i=1,4,6$; (\ref{lw}), (\ref{5lw}) and (\ref{s633}) generates a CEA.
\end{pro}

\begin{rk} We note that the cases $\{(1_1), (2_4), (1_6)\}$, $\{(2_1), (1_4), (1_6)\}$ are similar to the case $\{(1_1), (1_4), (2_6)\}$ and $\{(2_1), (2_4), (1_6)\}$, $\{(2_1), (1_4), (2_6)\}$ are similar to the case $\{(1_1), (2_4), (2_6)\}$. Hence these cases do not give any new CEA.
\end{rk}

\section{Dynamical properties of CEAs given by matrix (\ref{M})}

In this section we consider CEAs constructed by matrix (\ref{M}). To avoid many special cases we assume
\begin{equation}\label{det}
\det(\M^{[s,t]})=a_1^{[s,t]}a_4^{[s,t]}a_6^{[s,t]}\ne 0  \ \ \mbox{for all} \ \ (s,t).
\end{equation}

We note that the CEAs given in Proposition \ref{pr1} satisfy the condition (\ref{det}).

In \cite{CLR} a notion of property
transition for CEAs is defined. We recall the definitions:

\begin{defn}\label{d4} Assume a CEA, $E^{[s,t]}$, has a property, say $P$,
at pair of times $(s_0,t_0)$; one says that the CEA has $P$ property
transition if there is a pair $(s,t)\ne (s_0,t_0)$ at which the CEA
has no the property $P$.
\end{defn}

Denote
$$\mathcal T=\{(s,t): 0\leq s\leq t\};$$
$$\mathcal T_P=\{(s,t)\in \mathcal T: E^{[s,t]} \ \ \mbox{has property} \  P \};$$
$$\mathcal T_P^0=\mathcal T\setminus \mathcal T_P=\{(s,t)\in \mathcal T: E^{[s,t]} \ \ \mbox{has no property} \ P \}.$$

The sets have the following meaning

$\mathcal T_P$-the duration of the property $P$;

$\mathcal T_P^0$-the lost duration of the property $P$;

A {\it character} for an algebra $A$ is a nonzero multiplicative
linear form on $A$, that is, a nonzero algebra homomorphism from $A$
to $\R$ \cite{ly}. Not every algebra admits a character. For
example, an algebra with the zero multiplication has no character.

\begin{defn}\label{d3} A pair $(A, \sigma)$ consisting of an algebra $A$ and a
character $\sigma$ on $A$ is called a {\it baric algebra}. The
homomorphism $\sigma$ is called the weight (or baric) function of
$A$ and $\sigma(x)$ the weight (baric value) of $x$.
\end{defn}

Recall that the element $x=\sum_{i=1}^nx_ie_i$ of an algebra $A$ is called an {\it absolute
nilpotent} if $x^2=0$. For $n$-dimensional evolution algebra $x^2=0$ is given by the
following system
\begin{equation}\label{n1}
\sum_{i=1}^na_{ij}x_i^2=0, \ \ j=1,\dots,n.
\end{equation}

An element $x$ of an algebra $\A$ is called {\it idempotent} if
$x^2=x$.  We denote by ${\mathcal Id}(\A)$
the set of idempotent elements of an algebra $\A$. For an evolution algebra the equation $x^2=x$ can be written as
\begin{equation}\label{v1}
x_j= \sum_{i=1}^na_{ij}x_i^2, \ \ j=1,\dots,n.
\end{equation}

Now we are ready to formulate the following

\begin{thm}\label{ttb} If a CEA, $E^{[s,t]}$ is defined by a matrix (\ref{M}) satisfying (\ref{det}) then
\begin{itemize}
\item[1.] The algebra $E^{[s,t]}$ is baric for any $(s,t)\in \mathcal T$.
\item[2.] The algebra $E^{[s,t]}$ has a unique absolute
nilpotent element $x=0$ for any $(s,t)\in \mathcal T$.
\item[3.]
$${\mathcal Id}\left(E^{[s,t]}\right)=\{\lambda_1,\, \lambda_2\}\bigcup \left\{\begin{array}{ccc}
\{\lambda_3\}, & \mbox{if} & (s,t)\in \left\{(s,t)\in\mathcal T: D_1(s,t)=0\right\}\\[2mm]
\{\lambda_4,\lambda_5\},& \mbox{if} & (s,t)\in \left\{(s,t)\in\mathcal T: D_1(s,t)>0\right\}
\end{array}\right.
\bigcup $$ $$
\left\{\begin{array}{cccc}
\{\lambda_6\}, & \mbox{if}& (s,t)\in \left\{(s,t)\in\mathcal T: D_2(s,t)=D_3(s,t)=0\right\}\\[2mm]
\{\lambda_7,\lambda_8\},& \mbox{if} & (s,t)\in \left\{(s,t)\in\mathcal T: D_2(s,t)=0, \, D_3(s,t)>0\right\}\\[2mm]
\{\lambda_9\}, & \mbox{if} & (s,t)\in \left\{(s,t)\in\mathcal T: D_2(s,t)>0, \, D_4(s,t)=0\right\}\\[2mm]
\{\lambda_{10},\lambda_{11}\}, & \mbox{if} & (s,t)\in\left\{(s,t)\in\mathcal T: D_2(s,t)>0, \, D_4(s,t)>0\right\}\\[2mm]
\{\lambda_{12}\}, & \mbox{if}& (s,t)\in \left\{(s,t)\in\mathcal T: D_2(s,t)>0,\, D_5(s,t)=0\right\}\\[2mm]
\{\lambda_{13},\lambda_{14}\}, & \mbox{if} & (s,t)\in\left\{(s,t)\in\mathcal T: D_2(s,t)>0, \, D_5(s,t)>0\right\},
\end{array}\right.
$$
where $\lambda_1=(0, 0, 0)$, $\lambda_2=\left(0, 0, 1/a_6^{[s,t]}\right)$,
$\lambda_3=\left(0, 1/a_4^{[s,t]}, 1/(2a_6^{[s,t]})\right)$,
$$D_1(s,t)=1-{4a_5^{[s,t]}a_6^{[s,t]}\over (a_4^{[s,t]})^2},\ \ \lambda_{4,5}=\left(0, {1\over a_4^{[s,t]}}, {1\pm\sqrt{D_1(s,t)}\over 2a_6^{[s,t]}}\right),$$
$$D_2(s,t)=1-{4a_2^{[s,t]}a_4^{[s,t]}\over (a_1^{[s,t]})^2},\ \
D_3(s,t)=1-4a_6^{[s,t]}\left({a_3^{[s,t]}\over (a_1^{[s,t]})^2}+{a_5^{[s,t]}\over (2a_4^{[s,t]})^2}\right),$$
$$\lambda_6=\left({1\over a_1^{[s,t]}}, {1\over 2a_4^{[s,t]}}, {1\over 2a_6^{[s,t]}}\right),\ \
\lambda_{7,8}=\left({1\over a_1^{[s,t]}}, {1\over 2a_4^{[s,t]}}, {1\pm\sqrt{D_3(s,t)}\over 2a_6^{[s,t]}}\right),
$$
$$D_{4,5}(s,t)=1-4a_6^{[s,t]}\left({a_3^{[s,t]}\over (a_1^{[s,t]})^2}+a_5^{[s,t]}y_{4,5}^2\right), \ \ \mbox{with} \ \
y_{4,5}={1\pm\sqrt{D_2(s,t)}\over 2 a^{[s,t]}_4},$$
$$\lambda_9=\left({1\over a_1^{[s,t]}}, y_4, {1\over 2a_6^{[s,t]}}\right),\ \
\lambda_{10,11}=\left({1\over a_1^{[s,t]}}, y_4, {1\pm\sqrt{D_4(s,t)}\over 2a_6^{[s,t]}}\right),
$$
$$\lambda_{12}=\left({1\over a_1^{[s,t]}}, y_5, {1\over 2a_6^{[s,t]}}\right),\ \
\lambda_{13,14}=\left({1\over a_1^{[s,t]}}, y_5, {1\pm\sqrt{D_5(s,t)}\over 2a_6^{[s,t]}}\right).
$$
\end{itemize}
\end{thm}
\begin{proof} 1. Under condition (\ref{det}) this assertion follows from \cite[Theorem 3.2]{CLR}.

2. From condition (\ref{det}) and system (\ref{n1}) we conclude that $x=(0,0,0)$ is the unique solution.

3. For matrix (\ref{M}) the system (\ref{v1}) can be written as
\begin{equation}\label{xyz}\begin{array}{lll}
x_1=a_1^{[s,t]}x_1^2\\[3mm]
x_2=a_2^{[s,t]}x_1^2+a_4^{[s,t]}x_2^2\\[3mm]
x_3=a_3^{[s,t]}x_1^2+a_5^{[s,t]}x_2^2+a_6^{[s,t]}x_3^2.
\end{array}
\end{equation}
Hence for $x_1$ there is two possibilities $x_1=0$ and $x_1=1/a_1^{[s,t]}$. For each value of $x_1$ the $x_2$ also has up to two possible
values. For fixed values of $x_1$ and $x_2$ we have two possible values for $x_3$.
Carefully computing the values gives the result.
\end{proof}

\begin{cor} The CEA considered in Theorem \ref{ttb} has the following dynamics:
\begin{itemize}
\item[1.] It has not "baric" property transition.
\item[2.] It has not "uniqueness absolute nilpotent element" property transition.
\item[3.] It has "a fixed number of idempotent elements" property
transition.  The transition critical sets are given by $\{(s,t): D_i(s,t)=0\}$, $i=1,2,3,4,5$, i.e. the number of idempotent elements changes depending on time, when the time point crosses these sets.
\end{itemize}
\end{cor}

\section{CEAs corresponding to symmetric matrices}
In this section we consider $\M^{[s,t]}=\left(a_{ij}^{[s,t]}\right)_{i,j=1,\dots,n}$ with $a_{ij}^{[s,t]}=a_{ji}^{[s,t]}$ and solve the equation (\ref{k2}) for such matrices.
The equation (\ref{k2}) has the following form
\begin{equation}\label{sm1}
\sum_{k=1}^na_{ik}^{[s,\tau]}a_{kj}^{[\tau,t]}=a_{ij}^{[s,t]}, \ \ i,j=1,\dots,n.
\end{equation}
We introduce the following functions
\begin{equation}\label{sm2}
f_{i}(s,t)=\sum_{j=1}^na_{ij}^{[s,t]}, \ \ i=1,\dots,n.
\end{equation}
\begin{equation}\label{sm3}
g_{ik}(s,t)=\sum_{j=1}^na_{ij}^{[s,t]}-a_{ik}^{[s,t]}, \ \ i,k=1,\dots,n.
\end{equation}

Using (\ref{sm1}) from (\ref{sm2}) we get
$$f_{i}(s,t)=\sum_{j=1}^n\sum_{k=1}^na_{ik}^{[s,\tau]}a_{kj}^{[\tau,t]}=
\sum_{k=1}^n\left(a_{ik}^{[s,\tau]}\sum_{j=1}^na_{kj}^{[\tau,t]}\right)=
\sum_{k=1}^na_{ik}^{[s,\tau]}f_k(\tau,t), \ \ \, i=1,\dots,n.$$
Consequently, using symmetry of the matrix we get
$$\sum_{i=1}^nf_{i}(s,t)=\sum_{k=1}^n\left(\sum_{i=1}^na_{ik}^{[s,\tau]}\right)f_k(\tau,t)=\sum_{k=1}^nf_k(s,\tau)f_k(\tau,t) .$$
Hence
\begin{equation}\label{sm4}
\sum_{i=1}^n\left(f_{i}(s,t)-f_i(s,\tau)f_i(\tau,t)\right)=0.
\end{equation}
Now we shall derive an equation for functions $g_{ik}$.
 Using (\ref{sm1}) from (\ref{sm3}) we get
$$g_{ik}(s,t)=\sum_{j=1}^n\sum_{p=1}^na_{ip}^{[s,\tau]}a_{pj}^{[\tau,t]}-\sum_{p=1}^na_{ip}^{[s,\tau]}a_{pk}^{[\tau,t]}=
\sum_{p=1}^n\left(a_{ip}^{[s,\tau]}\left(\sum_{j=1}^na_{pj}^{[\tau,t]}-a_{pk}^{[\tau,t]}\right)\right).$$
Hence
$$g_{ik}(s,t)=\sum_{p=1}^na_{ip}^{[s,\tau]}g_{pk}(\tau,t), \ \ i,k=1,\dots,n.$$
Using the symmetry of the matrix from the last equality we get
$$\sum_{i=1}^ng_{ik}(s,t)=\sum_{p=1}^n\left(\sum_{i=1}^na_{ip}^{[s,\tau]}\right)g_{pk}(\tau,t)=
\sum_{p=1}^nf_p(s,\tau)g_{pk}(\tau,t), \ \ k=1,\dots,n.$$
Thus we obtain
\begin{equation}\label{sm5}
\sum_{i=1}^n\left(g_{ik}(s,t)-f_i(s,\tau)g_{ik}(\tau,t)\right)=0, \ \ k=1,\dots,n.
\end{equation}
We note that (\ref{sm4}) and (\ref{sm5}) are general equations for functions $f_i$ and $g_{ik}$. These equations are satisfied, in particular, if
\begin{equation}\label{sm6}
f_{i}(s,t)=f_i(s,\tau)f_i(\tau,t),  \ \ \mbox{for any} \ \ i=1,\dots,n;
\end{equation}
\begin{equation}\label{sm7}
g_{ik}(s,t)=f_i(s,\tau)g_{ik}(\tau,t), \ \ \mbox{for any} \ \ i,k=1,\dots,n.
\end{equation}
Equations (\ref{sm6}) have the following solutions
$$f_i(s,t)={\Phi_i(t)\over\Phi_i(s)}, \ \ i=1,\dots,n,$$
where $\Phi_i$ are arbitrary functions with $\Phi_i(s)\ne 0$.
Substituting this solution in (\ref{sm7})
we obtain
$$g_{ik}(s,t)={\gamma_{ik}(t)\over \Phi_i(s)}, \ \ i,k=1,\dots,n,$$
where $\gamma_{ik}$ are arbitrary functions.
By (\ref{sm2}) and (\ref{sm3}) we have
\begin{equation}\label{sm8}
a_{ik}^{[s,t]}=a_{ki}^{[s,t]}=f_i(s,t)-g_{ik}(s,t)={\Phi_i(t)-\gamma_{ik}(t)\over \Phi_i(s)}, \ \ i,k=1,\dots,n.
\end{equation}
Thus we have proved the following
\begin{thm} A matrix $\M^{[s,t]}=\left(a_{ik}^{[s,t]}\right)_{i,k=1,\dots,n}$
given by (\ref{sm8}) generates an $n$-dimensional CEA.
\end{thm}

\section{CEAs corresponding to block diagonal matrices}
Recall that a block diagonal matrix is a block matrix which is a square matrix, and having main diagonal blocks square matrices, such that the off-diagonal blocks are zero matrices. A block diagonal matrix $M$ has the form
$$M=\left(\begin{array}{cccc}
M_1&0&\cdots&0\\[3mm]
0&M_2&\cdots&0\\[3mm]
\vdots&\vdots&\ddots&\vdots\\[3mm]
0&0&\cdots&M_m\\[3mm]
\end{array}
\right),$$
where $M_k$ is a square matrix. It can also be indicated as $M_1\oplus M_2\oplus\dots\oplus M_m$.

\begin{lemma}\label{l1}
Let $\M^{[s,t]}=\M_1^{[s,t]}\oplus \M_2^{[s,t]}\oplus\dots\oplus \M_m^{[s,t]}$ be a block diagonal matrix. This satisfies
(\ref{k2}) if and only if each $\M_i^{[s,t]}$, $i=1,2,\dots,m$ satisfies the equation (\ref{k2}).
\end{lemma}
\begin{proof} It follows from the fact that the product of diagonal matrices
amounts to simply multiplying corresponding diagonal elements together.
\end{proof}

In \cite{CLR}, \cite{RM} it were constructed several CEAs. Using these CEAs and the CEAs constructed in
the previous sections, by Lemma \ref{l1} one can construct new chains of arbitrary dimensional evolution algebras, i.e.
the following is true

\begin{thm}\label{tb} If $\M^{[s,t]}=\M_1^{[s,t]}\oplus \M_2^{[s,t]}\oplus\dots\oplus \M_m^{[s,t]}$ is a block diagonal matrix, where $\M_i^{[s,t]}$ correspond to a known CEA. Then $\M^{[s,t]}$ generates a CEA.
\end{thm}

We note that the set of CEAs is very rich: taking matrices $\M^{[s,t]}$ constructed in Theorem \ref{tb} as blocks of a new matrix $\tilde{\M}^{[s,t]}$ one can construct new CEAs. Repeating the argument one can construct very rich class of CEAs.

\section*{ Acknowledgements} The work is supported by the
Grant No.0251/GF3 of Education and Science Ministry of Republic of Kazakhstan.

{}
\end{document}